\documentclass[a4paper]{article}

\usepackage{bbm}
\usepackage{amsmath}
\usepackage{amssymb}
\usepackage{amsfonts}
\usepackage{amsthm}
\usepackage{amscd}
\usepackage[english]{babel}
\usepackage[latin1]{inputenc}
\usepackage[T1]{fontenc}
\usepackage{lmodern}
\usepackage{fancyhdr, fancybox}
\usepackage{extramarks}
\usepackage{array}
\usepackage{multirow}
\usepackage[hang,symbol]{footmisc}
\usepackage{color}
\usepackage{graphicx}
\usepackage{caption}
\usepackage{nicefrac}
\usepackage{enumitem}

\usepackage{tikz}
\usetikzlibrary{matrix,arrows,positioning}
\usepackage{hyperref}
\usepackage[intoc]{nomencl}
\usepackage{hyphsubst}
\usepackage{makeidx}
\usepackage{relsize}

\usepackage[shortcuts]{extdash}

\usepackage{ellipsis}

\hypersetup{pdfborder={0 0 0}}

\newtheorem{thm}{Theorem}[section]
\newtheorem{lem}[thm]{Lemma}

\newtheorem{prop}[thm]{Proposition}
\newtheorem*{thm*}{Theorem}

\theoremstyle{definition}
\newtheorem{example}[thm]{Example}
\newtheorem{Def}[thm]{Definition}

\newtheorem{rmk}[thm]{Remark}

\usepackage{environ}
\NewEnviron{hiddenproofs}{}

\newcommand{\ses}[5]{ %
0 \to #1 \overset{#4}{\longrightarrow} #2 \overset{#5}{\longrightarrow} #3 \to 0 %
}

\renewcommand{\phi}{\varphi}

\def\haken#1{\underline{#1}{\raise -0.3ex\hbox{\vphantom{$#1$}\vrule}}\hspace{0.2em}}

\DeclareMathOperator{\im}{im}

\DeclareMathOperator{\rad}{Rad}
\DeclareMathOperator{\irr}{Irr}

\DeclareMathOperator{\module}{mod}
\renewcommand{\mod}{\module}
\DeclareMathOperator{\Mod}{Mod}
\DeclareMathOperator{\Hom}{Hom}
\DeclareMathOperator{\Aut}{Aut}
\DeclareMathOperator{\Ext}{Ext}
\renewcommand{\hom}{\Hom}

\DeclareMathOperator{\res}{res}
\DeclareMathOperator{\ind}{ind}

\DeclareMathOperator{\ql}{ql}

\DeclareMathOperator{\proj}{Proj}

\DeclareMathOperator{\maxspec}{Maxspec}
\DeclareMathOperator{\sl2}{\mathfrak{sl}(2)}
\DeclareMathOperator{\Usl2}{U_0(\mathfrak{sl}(2))}

\DeclareMathOperator{\id}{id}

\DeclareMathAlphabet{\mathpzc}{OT1}{pzc}{m}{it}
\DeclareMathOperator{\fun}{\mathpzc{Fun}}
\DeclareMathOperator{\funop}{\mathpzc{Fun}^{op}}
\DeclareMathOperator{\mmod}{mmod}

\newcommand{\N}{\mathbb{N}}
\newcommand{\Z}{\mathbb{Z}}
\renewcommand{\P}{\mathbb{P}}
\newcommand{\V}{\mathcal{V}}

\newcommand{\G}{\mathcal{G}}
\newcommand{\cN}{\mathcal{N}}
\newcommand{\cM}{\mathcal{M}}
\newcommand{\cH}{\mathcal{H}}

\newcommand{\cC}{\mathcal{C}}

\newcommand{\cZ}{\mathcal{Z}}
\newcommand{\cB}{\mathcal{B}}
\newcommand{\cO}{\mathcal{O}}
\newcommand{\cK}{\mathcal{K}}
\newcommand{\fm}{\mathfrak{m}}

\newcommand{\hideproofs}{\let\proof\hiddenproofs
	\let\endproof\endhiddenproofs}

\newcommand{\sfrac}[2]{%
  \hbox{\kern 0.1em%
  \raise 0.5ex\hbox {\smaller$#1$}%
  \kern -0.1em {$/$}%
  \kern -0.15em%
  \lower 0.25ex\hbox {\smaller$#2$}}%
  \kern  0.2em}

\title{\large{AR-Components of domestic finite group schemes:\\McKay-Quivers and Ramification}}
\author{\small Dirk Kirchhoff}
\date{}

\begin{document}
\maketitle

\begin{abstract} 
For a domestic finite group scheme, we give a direct description of the Euclidean components in its Auslander-Reiten quiver
via the McKay-quiver of a finite linearly reductive subgroup scheme of $SL(2)$. Moreover, for a normal subgroup scheme $\cN$
of a finite group scheme $\G$, we show that there is a connection between the ramification indices of the restriction morphism
$\P(\V_\cN)\rightarrow\P(\V_\G)$ between their projectivized cohomological support varieties and the ranks of the tubes
in their Auslander-Reiten quivers.
\end{abstract}

\section*{Introduction}
The Auslander-Reiten quiver of a self-injective finite-dimensional algebra is a powerful tool for understanding its representation theory
via combinatorical invariants.
In this work we are mainly interested in the Auslander-Reiten components of a group algebra $k\G$ of domestic representation type
for a finite group scheme $\G$ over an algebraically closed field $k$. \\
We say that an algebra $A$ has tame representation type if it possesses infinitely many isomorphism classes of indecomposable modules
and if in each dimension almost all isomorphism classes of indecomposable modules occur in only finitely many one-parameter families.
If additionally the number of one-parameter families is uniformly bounded, we say that $A$ has domestic representation type.
The group algebra $k\G:=k[\G]^*$ of a finite group scheme $\G$ is the dual of its coordinate ring. We say that $\G$ is domestic
if its group algebra is domestic.\\
In \cite{farnsteiner2012extensions} Farnsteiner classified the domestic finite group schemes 
over an algebraically closed field of characteristic $p>2$. Any such group scheme can be associated to a so called amalgamated
polyhedral group scheme. Moreover, the non-simple blocks of an amalgamated polyhedral group scheme are Morita-equivalent
to a radical square zero
tame hereditary algebra. In this way the components of the Auslander-Reiten quiver of these group schemes
are classified abstractly. Our goal is to describe these components in a direct way by using tensor products, McKay-quivers and
ramification indices of certain morphisms.\\

We will start by describing the Euclidean components. For this purpose, we show how to extend certain almost split sequences
from a normal subgroup scheme $\cN\subseteq \G$ to almost split sequences of $\G$ if the group scheme $\G/\cN$
is linearly reductive. Moreover, for a simple $\G/\cN$-module $S$ we will show that the tensor functor $-\otimes_kS$
sends these extended almost split sequences to almost split sequences. In section 2 we will use these results to show that
for any amalgamated polyhedral group scheme $\G$ there is a finite linearly reductive subgroup scheme $\tilde{\G}\subseteq SL(2)$ 
such that the Euclidean components
of $\Gamma_s(\G)$ can be explicitly described by the McKay-quiver $\Upsilon_{L(1)}(\tilde{\G})$. 

Of great importance for the proofs of these results is the fact that the category of $\G$-modules is closed under taking tensor products
of $\G$-modules. This comes into play in the definition of the McKay-quiver, for the construction of new almost split sequences and
in the description of the Euclidean components. Thanks to this property, we are also able to introduce geometric invariants for the representation theory of $\G$.
If $\G$ is any finite group scheme, one can endow the even cohomology ring $H^\bullet(\G,k)$ with the structure of a commutative graded
$k$-algebra. Thanks to the Friedlander-Suslin-Theorem (\cite{friedlander1997cohomology}), this algebra is finitely-generated.
Therefore, the maximal ideal spectrum $\V_\G$ of $H^\bullet(\G,k)$ is an affine variety. As $H^\bullet(\G,k)$ is graded,
we can also consider its projectivized variety $\P(\V_\G)$. \\
Now let us again assume that $\cN\subseteq \G$ is a normal subgroup scheme such that $\G/\cN$ is linearly reductive. Then the ramification indices
of the restriction morphism $\P(\V_\cN)\rightarrow\P(\V_\G)$ will give upper bounds for ranks of the corresponding tubes in the
Auslander-Reiten quiver. Here a tube $\Z/(r)[A_\infty]$ of rank $r$ can be regarded as a quiver which is arranged on an infinite tube
with circumference $r$. Moreover, if $\G$ is an amalgamated polyhedral group scheme and $\cN=\G_1$ is its first Frobenius kernel, the ranks are equal to
the corresponding ramification indices. Altogether we will proof the following:
\begin{thm*}
Let $\G$ be an amalgamated polyhedral group scheme and $\Theta$ a component of the stable Auslander-Reiten quiver $\Gamma_s(\G)$. Then the following hold:
\begin{enumerate}[label={(\roman*)}]
\item If $\Theta$ is Euclidean, then there is a component $Q$ of the separated quiver $\Upsilon_{L(1)}(\tilde{\G})_s$ and a concrete
isomorphism $\Theta\cong\Z[Q]$.
\item Let $\Theta$ be a tube and $e_\Theta$ the ramification index of the restriction morphism $\P(\V_{\G_1})\rightarrow\P(\V_\G)$ at the
corresponding point $x_\Theta$. Then $\Theta\cong \Z/(e_\Theta)[A_\infty]$. 
\end{enumerate}
\end{thm*}

 There seems to be a connection to a result of Crawley-Boevey (\cite{crawley1988tame}), which states
that a finite-dimensional tame algebra has only finitely many non\-/homogeneous tubes. On the other side, the restriction morphism $\P(\V_\cN)\rightarrow\P(\V_\G)$
is finite and has constant ramification on an open dense subset of $\P(\V_\cN)$. 
Therefore, there are only finitely many exceptional ramification points.
In our situation, all but finitely many points will be unramified and 
a tube can only be non-homogeneous, if it belongs to the image of a ramification point.

\section{Induction of almost split sequences}
Let $k$ be an algebraically closed field.
Given a finite group scheme $\G$ with normal subgroup scheme $\cN$, we want to investigate conditions under
which the induction functor $\ind_\cN^\G:\mod\cN\rightarrow\mod\G$ sends an almost split exact sequence to a direct sum of
almost split exact sequences.
We start by giving an overview of the functorial approach to almost split sequences. For further details we refer the reader to 
\cite[IV.6]{assem2006elements}.
Let $A$ be a finite-dimensional $k$-algebra. Denote by $\funop A$ and $\fun A$ the categories of contravariant and covariant
$k$-linear functors from $\mod A$ to $\mod k$. A functor $F$ in $\funop A$ is finitely generated if the functor $F$
is isomorphic to a quotient of $\hom_A(-,M)$ for some $M\in\mod A$. A functor $F$ in $\funop A$ is finitely presented if there is an exact sequence
\[\hom_A(-,M)\rightarrow\hom_A(-,N)\rightarrow F\rightarrow 0\]
of functors in $\funop A$ for some $M,N\in\mod A$. 
The full subcategory of $\funop A$ consisting of the finitely presented functors will be denoted by
$\mmod A$.
Up to isomorphism the finitely generated
projective functors in $\funop A$ are exactly the functors of the form
$\hom_A(-,M)$. Such a functor is indecomposable if and only if the $A$-module $M$ is
indecomposable.\\
For $A$-modules $M$ and $N$ we define the radical of $\hom_A(M,N)$ as 
\[\rad_A(M,N):=\{\phi\in\hom_A(M,N)\;\vert\;\phi\text{ is not an isomorphism}\}.\]
Then $\rad_A(-,M)$ is a subfunctor of $\hom_A(-,M)$ and we define the functor $S^M:=\hom_A(-,M)/\rad_A(-,M)$.
Up to isomorphism the simple functors in $\funop A$ are exactly the functors of the form
 $S^M$ with an indecomposable $A$-module $M$.
 The projective cover of $S^M$ is $\hom_A(-,M)$.
 Let $N$ be an indecomposable $A$-module. An $A$-module homomorphism $g:M\rightarrow N$ is (minimal) right almost split if
 and only if the induced exact sequence
 \[\hom_A(-,M)\rightarrow\hom_A(-,N)\rightarrow S^N\rightarrow 0\]
 of functors in $\funop A$ is a (minimal) projective presentation of $S^N$. A functor $F:\mod A\rightarrow \mod B$
 induces a functor $F:\mmod A\rightarrow \mmod B$ via $F(\hom_A(-,M))=\hom_B(-,F(M))$.
 There are dual notions and results for left almost split morphisms and
 functors in $\fun A$.\\
 
 In \cite[3.8]{reiten1985skew}  Reiten and Riedtmann used this functorial approach to show
 that for a skew group algebra $A*G$, with a finite group $G$, the induction functor $\ind_1^G:\mod A\rightarrow \mod A*G$ and the restriction functor
 $\res_1^G:\mod A*G\rightarrow \mod A$ send almost split sequences
 to direct sums of almost split sequences if the group order of $G$ is invertible in $k$. 
 Their proof relied on the following properties of the involved functors:
 \begin{itemize}
 	\item[(A)] 
 	\begin{enumerate}[label={(\roman*)}]
 		\item There is a split monomorphism of functors $\id_{\mod A}\rightarrow \res_1^G\ind_1^G$.
 		\item There is a split epimorphism of functors $\ind_1^G\res_1^G\rightarrow \id_{\mod A*G}$.
 	\end{enumerate}
 	\item[(B)] $(\ind_1^G,\res_1^G)$ and $(\res_1^G,\ind_1^G)$ are adjoint pairs of functors.
 	\item[(\~C)] There is a finite group $G$ acting on $\mod A$ such that for every $A$-module $M$ there is a
 	decomposition $\res_1^G\ind_1^G M=\bigoplus_{g\in G}M^g$ and if $\phi:M\rightarrow N$ is $A$-linear, then
 	$\res_1^G\ind_1^G(\phi)=(g.\phi)_{g\in G}:\bigoplus_{g\in G}M^g\rightarrow\bigoplus_{g\in G}N^g$.
 \end{itemize}
 In \cite[3.5]{reiten1985skew} it was shown that these properties also hold for the induced functors $\ind_1^G:\mmod A\rightarrow\mmod A*G$
 and $\res_1^G:\mmod A*G\rightarrow \mmod A$ and that they imply the following property:
 \begin{itemize}
 	\item[(C)] $\ind_1^G:\mmod A\rightarrow\mmod A*G$
 	and $\res_1^G:\mmod A*G\rightarrow \mmod A$ preserve semisimple objects and projective covers.
 \end{itemize}
 We want to apply the ideas of the proof in the context of group algebras of finite group schemes.
 If $\G$ is a finite group scheme, then the dual of its coordinate ring $k\G:=(k[\G])^*$ is called
 the group algebra of $\G$. In this situation we will not always have analogous results for the induction and restriction functor. For example, in
   \cite[3.1.4]{farnsteiner2009group} it was already shown that for the restriction functor this is possible if and only if the
   ending term of the almost split sequence fulfils a certain regularity property. \\
   Let $\cN$ be a normal subgroup scheme of $\G$. The $\cN$-modules which will be of our interest are restrictions of
   $\G$-modules. To obtain an analogue of property (\~C) we will use the following result:
   \begin{lem}\label{inducing a G-module}
   Let $\G$ be a finite group scheme and $\cN$ be a normal subgroup scheme of $\G$.
   Let $M$ be a $\G$-module. Then there is a $\G$-linear isomorphism
   \[\psi_M:\ind_\cN^\G\res_\cN^\G M\rightarrow M\otimes_kk(\G/\cN)\]
   which is natural in $M$.\\
   In particular, $\res_\cN^\G\ind_\cN^\G M\cong M^n$ where $n=\dim_kk(\G/\cN)$.
   \end{lem}
   \begin{proof}
   This follows directly from the tensor identity. Alternatively, exactly as in the proof of \cite[5.1]{kirchhoff2015domestic} one can show that the map
   		\[\psi_M:\ind_\cN^\G M\rightarrow M\otimes_kk(\G/\cN),\;a\otimes m\mapsto \sum_{(a)}a_{(1)}m\otimes \pi(a_{(2)})\]
   		is an isomorphism of $\G$-modules, where $\pi:k\G\rightarrow k(\G/\cN)$ is the canonical projection.
   		A direct computation shows that it is natural in $M$.
   \end{proof}

  \begin{prop}\label{ind is direct sum of AR-seq}
  	 Let $\G$ be a finite group scheme and $\cN$ be a normal subgroup scheme such that $\G/\cN$ is linearly reductive. 
  	 Let $X, Y$ and $E$ be $\cN$-modules
  	 such that
  	 \begin{enumerate}
  	 	\item every indecomposable direct summand of the modules $X, Y$ and $E$ is the restriction of a $\G$-module, and
  	 	\item $\mathcal{E}:\ses{X}{E}{Y}{\phi}{\psi}$ is an almost split exact sequence of $\cN$-modules.
  	 \end{enumerate}
  	 Then $\ind_\cN^\G\mathcal{E}$ is a direct sum of almost split exact sequences.
  \end{prop}
  \begin{proof}
  Denote by $\cC$ the full subcategory of $\mod \cN$ consisting of direct sums of indecomposable $\cN$-modules which are restrictions of $\G$-modules.
  Due to \ref{inducing a G-module}, the properties (A)(i) and (\~C) hold in $\cC$ if $G$ is a group acting trivially on $\cC$. 
    As $\G/\cN$ is linearly reductive, the $k(\G/\cN)$-Galois extension $k\G:k\cN$ is separable by \cite[3.15]{doi1989hopf}. This yields property (A)(ii).
    Thanks to \cite[1.7(5)]{takeuchi1981hopf}, the ring extension $k\G:k\cN$ is a free Frobenius extension of first kind, i.e.
    \begin{enumerate}[label={(\alph*)}]
    \item $k\G$ is a finitely generated free $k\cN$-module, and
    \item there is a $(k\G,k\cN)$-bimodule isomorphism $k\G\rightarrow\hom_\cN(k\G,k\cN)$.
    \end{enumerate}
    Hence, by \cite[2.1]{morita1965adjoint} the induction and coinduction functors are equivalent, so that property (B) holds.\\
  	As $\psi$ is minimal right almost split, the exact sequence
  	\[\hom_\cN(-,E)\rightarrow\hom_\cN(-,Y)\rightarrow S^{Y}\rightarrow 0\]
  	is a minimal projective presentation of $S^Y$. We can now apply the arguments of the proof
  	of \cite[3.6]{reiten1985skew} by keeping the following things in mind:
  	\begin{itemize}
  		\item The simple functors which occur in our situation are all of the form $S^V$ with $V$ being 
  		an indecomposable module in $\cC$.
  		\item The projective covers which occur in our situation are all of the form $\hom_\cN(-,V)$ with $V\in\cC$.
  	\end{itemize}
  	Therefore, the functor $\ind_\cN^\G S^Y$ is semisimple and the exact sequence
  	
  	\[\ind_\cN^\G\hom_\cN(-,E)\rightarrow\ind_\cN^\G\hom_\cN(-,Y)\rightarrow \ind_\cN^\G S^{Y}\rightarrow 0\]
  	is a minimal projective presentation of $\ind_\cN^\G S^Y\cong S^{\ind_\cN^\G Y}$.
  	Hence, $\ind_\cN^\G\psi$ is a direct sum of minimal right almost split homomorphisms. \\
  	Dually, one can show that $\ind_\cN^\G\phi$ is a direct sum of minimal left almost split homomorphisms.
  \end{proof}

Let $\G$ be a group scheme and $X$ and $M$ be $\G$-modules. We define the $k$-vector spaces
\[\rad_\G^2(X,M):=\{\alpha\;\vert\;\exists Z\in\Mod\G,\phi\in\rad_\G(X,Z),\psi\in\rad_\G(Z,M):
\alpha=\psi\circ\phi\}\]
and $\irr_\G(X,M):=\rad_\G(X,M)/\rad_\G^2(X,M)$.
\begin{lem}
	Let $\G$ be a reduced group scheme, $\cN$ be a normal subgroup scheme of $\G$ and
	$X$ and $M$ be $\G$-modules.
	Then $\rad_\cN(X,M)$ and $\rad_\cN^2(X,M)$ are $\G$-submodules of $\Hom_\cN(X,M)$.
\end{lem}
\begin{proof}
	As $\G$ is reduced, we only need to prove that $\rad_\cN(X,M)$ and $\rad_\cN^2(X,M)$ are
	 $\G(k)$-submodules of $\Hom_\cN(X,M)$ (see \cite[Remark after 2.8]{jantzen2007representations}). \\
	 Let $g\in\G(k)$ and $\phi\in\rad_\cN(X,M)$. Assume that $g.\phi$ is an isomorphism
	 with inverse $\psi$. Then $g^{-1}.\psi$ is an inverse of $\phi$, a contradiction. \\
	Let $g\in\G(k)$ and $\phi\in\rad_\cN^2(X,M)$. Then there is an $\cN$-module $Z$ and homomorphisms $\alpha\in\rad_\cN(X,Z)$, 
	$\beta\in\rad_\cN(Z,M)$ such that $\phi=\beta\circ\alpha$. Denote by $Z^g$ the $\cN$-module
	$Z$ with action twisted by $g^{-1}$, i.e. $h.z=h^{g^{-1}}z$ for $h\in k\cN$ and $z\in Z^g$.
	As $X$ and $M$ are $\G$-modules, we can define the $\cN$-linear maps
	\[\tilde{\alpha}:X\rightarrow Z^{g},\; x\mapsto \alpha(g^{-1}x)\text{ and }
	\tilde{\beta}:Z^{g}\rightarrow M,\; z\mapsto g\beta(z).\]
	Then we obtain $g.\phi=\tilde{\beta}\circ\tilde{\alpha}$.
	If $\tilde{\alpha}$ is an isomorphism with inverse $\gamma:Z^g\rightarrow X$, then
	the $\cN$-linear map
	\[\tilde{\gamma}:Z\rightarrow X,\;z\rightarrow g^{-1}\gamma(z)\]
	is an inverse of $\alpha$, a contradiction.
	In the same way, $\tilde{\beta}$ is not an isomorphism.
	Hence $g.\phi=\tilde{\beta}\circ\tilde{\alpha}\in\rad^2_\cN(X,M)$.
\end{proof}

\begin{prop}\label{extend ar sequence}
	Let $\G$ be a finite subgroup scheme of a reduced group scheme $\cH$ and $\cN\subseteq\G$ a normal subgroup scheme
	of $\cH$ such that $\G/\cN$ is linearly
	reductive. Let $X$ and $M$ be indecomposable $\G$-modules such that there is an almost split exact sequence
	\[\mathcal{E}:\ses{\tau_\cN(M)}{X^n}{M}{}{}\]
	of $\cN$-modules.
	Then there is a short exact sequence 
	\[\tilde{\mathcal{E}}:\ses{N}{X\otimes_k \irr_\cN(X,M)}{M}{}{}\]
	of $\G$-modules such that $\res_\cN^\G\tilde{\mathcal{E}}=\mathcal{E}$.
\end{prop}
\begin{proof}
	Consider the map
	\[\tilde{\psi}:X\otimes_k\hom_\cN(X,M)\rightarrow M,\; x\otimes f\mapsto f(x)\]
	of $\G$-modules. As $\G/\cN$ is linearly reductive, the short exact sequence
	\[\ses{\rad^2_\cN(X,M)}{\rad_\cN(X,M)}{\irr_\cN(X,M)}{}{}\]
	of $\G$-modules splits. This yields a decomposition
	\[\rad_\cN(X,M)\cong\rad^2_\cN(X,M)\oplus\irr_\cN(X,M).\]
	Hence there results a $\G$-linear map $\psi:X\otimes_k\irr_\cN(X,M)\rightarrow M$
	by restricting $\tilde{\psi}$ to $X\otimes_k\irr_\cN(X,M)$.
	Let $(f_i)_{1\leq i\leq n}:X^n\rightarrow M$ be the surjection given by $\mathcal{E}$.
	Since $\mathcal{E}$ is almost split, the maps $(f_i)_{1\leq i\leq n}$ form a $k$-basis of
	$\irr_\cN(X,M)$ (c.f. \cite[IV.4.2]{assem2006elements}). Therefore the restriction $\res_\cN^\G(\psi)$ equals $(f_i)_{1\leq i \leq n}$.
\end{proof}
In \ref{ind is direct sum of AR-seq} we have seen that if we apply the induction functor
to certain almost split sequences, they will be the direct sum of almost split sequences.
Our next result enables us to describe this decomposition under a certain regularity condition.\\

Let $\cM$ be a multiplicative group scheme, i.e. the coordinate ring of $\cM$ is isomorphic
to the group algebra $kX(\cM)$ of its character group. 
Then we obtain for any $\cM$-module $M$ a weight space decomposition
$M=\bigoplus_{\lambda\in X(\cM)}M_\lambda$ with
\[M_\lambda:=\{m\in M\:\vert\:hm=\lambda(h)m\;\text{for all}\;h\in k\cM\}.\]
Let $\G$ be a finite group scheme and
$\cN\subseteq\G^0$ a normal subgroup scheme of $\G$ such that $\G^0/\cN$ is linearly reductive.
By Nagata's Theorem \cite[IV,\S 3,3.6]{demazure1970groupes}, an infinitesimal group scheme is linearly reductive if
and only if it is multiplicative. \\
Let $M$ be a $\G^0$-module. For any $\lambda\in X(\G^0/\cN)$ we obtain a $\G^0$-module
$M\otimes_kk_\lambda$, the tensor product of $M$ with the one-dimensional $\G^0/\cN$-module defined
by $\lambda$. This defines an action of $X(\G^0/\cN)$ (here we identify $X(\G^0/\cN)$ with a subgroup of 
$X(\G^0)$ via the canonical inclusion $X(\G^0/\cN)\hookrightarrow X(\G^0)$ which is induced by the canonical
projection $\G^0\twoheadrightarrow \G^0/\cN$) on the isomorphism classes of $\G^0$-modules
and we define the stabilizer of a $\G^0$-module as
\[X(\G^0/\cN)_M:=\{\lambda\in X(\G^0/\cN)\:\vert\:M\otimes_kk_\lambda\cong M\}.\] A $\G^0$-module
$M$ is called $\G^0/\cN$-regular if its stabilizer is trivial.

\begin{prop}\label{decompose induced ar seq}
In the situation of \ref{extend ar sequence}  assume that $M$ and $N$ are $\G^0/\cN$-regular.
Let $S$ be a simple $\G/\cN$-module. Then the short exact sequence $\tilde{\mathcal{E}}\otimes_kS$ is almost split.
\end{prop}
\begin{proof}
Let $k(\G/\cN)=\bigoplus_{i=1}^nS_i$ be the decomposition into simple $\G/\cN$-modules.
		For any $\G$-module $U$, let $\psi_U:\ind_\cN^\G U\rightarrow U\otimes_kk(\G/\cN)$ be the natural isomorphism
		given by \ref{inducing a G-module}.
		We obtain the following commutative diagram with exact rows:
		\begin{center}
			\begin{tikzpicture}
			\matrix (m) [matrix of math nodes, row sep=3em, column sep=1em]
			{
				0 & \ind_\cN^\G N& \ind_\cN^\G (X\otimes_k\irr_\cN(X,M)) & \ind_\cN^\G M & 0 \\
				 0 & N\otimes_kk(\G/\cN)  &  X\otimes_k\irr_\cN(X,M)\otimes_kk(\G/\cN)  &  M\otimes_kk(\G/\cN) & 0 \\ };
			\path[-stealth]
			(m-1-1) edge (m-1-2)
			(m-1-2) edge (m-1-3)
			(m-1-3) edge (m-1-4)
			(m-1-4) edge (m-1-5)
			(m-2-1) edge (m-2-2)
			(m-2-2) edge (m-2-3)
			(m-2-3) edge (m-2-4)
			(m-2-4) edge (m-2-5)
			(m-1-2) edge node [right] {$\psi_N$} (m-2-2)
			(m-1-3) edge node [right] {$\psi_{X\otimes_k\irr_\cN(X,M)}$}(m-2-3)
			(m-1-4) edge node [right] {$\psi_M$}(m-2-4);
			
			\end{tikzpicture}
		\end{center}
		As all vertical arrows are isomorphisms, the two exact sequences $\ind_\cN^\G\mathcal{E}$ and
		$\bigoplus_{i=1}^n\tilde{\mathcal{E}}\otimes_kS_i$ are equivalent.
\\
		Thanks to \ref{ind is direct sum of AR-seq} the exact sequence $\ind_\cN^\G\mathcal{E}$
		is a direct sum of almost split sequences. 
		By \cite[5.1]{kirchhoff2015domestic}, the $\G$-modules
	 $M\otimes_kS$ and $N\otimes_kS$ are indecomposable. Moreover, the sequence $\tilde{\mathcal{E}}\otimes_kS$
	 does not split. Otherwise, the sequence $\res_\cN^\G(\tilde{\mathcal{E}}\otimes_kS)$ would split and therefore 
	also $\mathcal{E}$.
		Hence, the sequence $\tilde{\mathcal{E}}\otimes_kS$ is equivalent to an indecomposable direct summand of
		 $\ind_\cN^\G\mathcal{E}$. As the Krull-Schmidt theorem holds in the category of short exact sequences of 
		 finite-dimensional $\G$-modules (as the space of morphisms between those sequences is finite-dimensional, c.f. \cite{krause2014krull}),
		 it follows that $\tilde{\mathcal{E}}\otimes_kS_i$ is almost split.
\end{proof}
\section{The McKay and Auslander-Reiten quiver of finite domestic group schemes}
Let $k$ be an algebraically closed field of characteristic $p>2$. The group algebra $kSL(2)_1$ is isomorphic to the restricted universal 
enveloping algebra $\Usl2$ of the restricted Lie algebra $\sl2$. There are one-to-one correspondences between the representations of $SL(2)_1$, $\Usl2$ and
$\sl2$. The indecomposable representations of the restricted Lie algebra $\sl2$ were classified by Premet in \cite{premet1991green}.
For $d\in\N_0$ we consider the $(d+1)$-dimensional Weyl module $V(d)$ of highest weight $d$. Weyl modules are rational $SL(2)$-modules which are obtained
by twisting the $2$-dimensional standard module with the Cartan involution ($x\mapsto -x^{tr}$) and taking its $d$-th symmetric power.
For $d\leq p-1$ we obtain in this way exactly the simple $\Usl2$-modules $L(0),\dots,L(p-1)$.\\

Let $A$ be a self-injective $k$-algebra. We denote by $\Gamma_s(A)$ the stable Auslander-Reiten quiver of $A$.
The vertices of this valued quiver are the isomorphism classes of non-projective indecomposable $A$-modules and the arrows correspond
to the irreducible morphisms between these modules. Moreover, we have an automorphism $\tau_A$ of $\Gamma_s(A)$, called the Auslander-Reiten translation.
For a self-injective algebra the Auslander-Reiten translation is the composite $\nu\circ\Omega^2$, where $\nu$ denotes the Nakayama functor of $\mod A$ and
$\Omega$ the Heller shift of $\mod A$. For further details we refer to \cite{auslander1997representation} and \cite{assem2006elements}.\\
Let $Q$ be a quiver. We denote by $\Z[Q]$ the translation quiver with underlying set $\Z\times Q$, arrows
$(n,x)\rightarrow(n,y)$ and $(n+1,y)\rightarrow(n,x)$ for any arrow $x\rightarrow y$ in $Q$ and translation $\tau:\Z[Q]\rightarrow \Z[Q]$
given by $\tau(n,x)=(n+1,x)$.\\
Let $\Theta\subseteq\Gamma_s(A)$ be a connected component of the stable Auslander-Reiten 
quiver of $A$. Thanks to the Struktursatz of Riedtmann \cite{riedtmann1980algebren}, there
is an isomorphism of stable translation quivers $\Theta\cong\Z[T_\Theta]/\Pi$ where $T_\Theta$ denotes a directed tree and $\Pi$ is an admissible 
subgroup of $\Aut(\Z[T_\Theta])$. The underlying undirected tree $\overline{T}_\Theta$ is called the tree class of $\Theta$.

The Auslander-Reiten quiver of each block of $kSL(2)_1$ consists of two components of
type $\Z[\tilde{A}_{1,1}]$ and infinitely many homogeneous tubes $\Z[A_\infty]/(\tau)$. Thanks to \cite[4.1]{farnsteiner2009group},
each of the $p-1$ Euclidean
components $\Theta(i)$ contains exactly one simple $SL(2)_1$-module $L(i)$ with $0\leq i\leq p-2$. This component is then given by
\[\Theta(i)=\{\Omega^{2n}(L(i)),\Omega^{2n+1}(L(p-2-i))\;\vert\;n\in\Z\}\]
with almost split sequences
\[\ses{\Omega^{2n+2}(L(i))}{\Omega^{2n+1}(L(p-2-i))\oplus\Omega^{2n+1}(L(p-2-i))}{\Omega^{2n}(L(i))}{}{}.\]
In the following it will be convenient to have a common notation for the Weyl modules and their duals. We set
\[V(n,i):=
\begin{cases}
V(np+i) & \text{if }n\geq 0 \\
V(-np+i)^* & \text{if }n\leq 0 \\
\end{cases}\]
for $n\in\Z$ and $0\leq i\leq p-1$.\\
\begin{lem}\label{heller of weyl}
	Let $\G$ be a finite subgroup scheme of $SL(2)$ with $SL(2)_1\subseteq\G$ such that
	$\G/SL(2)_1$ is linearly reductive.
	Let $0\leq i\leq p-2$ and $n\in\Z$. Then
	\[\Omega^{2n}_\G(L(i))\cong V(2n,i)\]
	and
	\[
	\Omega^{2n+1}_\G(L(i))\cong V(2n+1,p-2-i)
	\]
	In particular, $\Theta(i)=\{V(n,i)\;\vert\; n\in\Z\}$.
\end{lem}
\begin{proof}
	By \cite[7.1.2]{xanthopoulos1992question}, there is for $n\leq 0$ a short exact sequence
	\[\ses{V(n,i)}{P(i)\otimes_k (V(n)^*)^{[1]}}{V(n-1,p-2-i)}{}{}\]
	of $SL(2)$-modules, where $(V(n)^*)^{[1]}$ denotes the first Frobenius twist of $V(n)^*$ (c.f. \cite[I.9.10]{jantzen2007representations}).
	As $P(i)$ is a projective $SL(2)$-module, we obtain that the $\G$-module $P(i)\otimes_k (V(n)^*)^{[1]}$ is projective.
	Hence there is a projective $\G$-module $P$ with $\Omega_\G(V(n-1,p-2-i))\oplus P\cong V(n,i)$.
	Since $V(n,i)$ is a non-projective indecomposable $\G$-module, we obtain
	 $\Omega_\G(V(n-1,p-2-i))\cong V(n,i)$.\\
	Dualizing the above sequence yields $\Omega_\G(V(n,i)))=V(n+1,p-2-i)$ for $n\geq 0$ in the same way.
\end{proof}

\begin{lem}\label{irr for SL(2)}
	For all $n\in\Z$ and $0\leq i\leq p-2$ the $SL(2)$-modules $L(1)^{[1]}$ and $\irr_{SL(2)_1}(V(n+1,i),V(n,i))$ are
	isomorphic.
\end{lem}
\begin{proof}
	As in the proof of \ref{heller of weyl} there is for all $n\in\Z$ a short exact sequence
	\[\ses{V(n+1,i)}{P}{V(n,p-2-i)}{}{}\]
	of $SL(2)$-modules with $P$ being projective over $SL(2)_1$. 
	Applying the functor $\hom_{SL(2)_1}(-,V(n,i))$ to this sequence yields the following exact sequence of $SL(2)$-modules:
	\begin{align*}
	&\hom_{SL(2)_1}(P,V(n,i))\overset{\alpha}{\rightarrow}
	\hom_{SL(2)_1}(V(n+1,i),V(n,i)) \\ \overset{\beta}{\rightarrow}
	&\Ext^1_{SL(2)_1}(V(n,p-2-i),V(n,i))\rightarrow
	\Ext^1_{SL(2)_1}(P,V(n,i))
	\end{align*}
	As $P$ is a projective $SL(2)_1$-module, we obtain that
	$\Ext_{SL(2)_1}^1(P,V(n,i))=(0)$. Hence the map $\beta$ is surjective.
	By \cite[2.4]{erdmann1995ext}, the $SL(2)$-modules $L(1)^{[1]}$ and $\Ext^1_{SL(2)_1}(V(n,p-2-i),V(n,i))$
	are isomorphic. As $\im\alpha=\ker\beta$ is contained in
	$\rad^2_{SL(2)_1}(V(n+1,i),V(n,i))$, we obtain a surjective morphism
	\[L(1)^{[1]}\rightarrow\irr_{SL(2)_1}(V(n+1,i),V(n,i)) \]
	of $SL(2)$-modules. Since both modules are $2$-dimensional, this map is an isomorphism.
\end{proof}

Let $\cH$ be a finite linearly reductive group scheme, $S_1,\dots,S_n$ a complete set of pairwise non-isomorphic
simple $\cH$-modules and $L$ be an $\cH$-module. For each $1\leq j\leq n$ there are $a_{ij}\geq 0$ such that
\[L\otimes_k S_j\cong \bigoplus_{i=1}^na_{ij}S_i.\]
The McKay quiver $\Upsilon_L(\cH)$ of $\cH$ relative to $L$ is the quiver with underlying set $\{S_1,\dots S_n\}$
and $a_{ij}$ arrows from $S_i$ to $S_j$.\\
Let $\G$ be a domestic finite group scheme. We denote by $\G_{lr}$ the largest linearly reductive normal
subgroup scheme of $\G$. Let $\cZ$ be the center of the group scheme $SL(2)$. Thanks to
\cite[4.3.2]{farnsteiner2012extensions}, there is a finite linearly reductive subgroup scheme $\tilde{\G}$
of $SL(2)$ with $\cZ\subseteq\tilde{\G}$ such that $\G/\G_{lr}$ is isomorphic to
$(SL(2)_1\tilde{\G})/\cZ$. Group schemes which arise in this fashion are also called amalgamated polyhedral group schemes. 
By the above, any domestic finite group scheme can be associated to one of these group schemes. 
If $\G$ is an amalgamated polyhedral group scheme we set $\hat{\G}:=SL(2)_1\tilde{\G}$.\\
For any Euclidean diagram $(\tilde{A}_n)_{n\in\N}, (\tilde{D}_n)_{n\geq 4}$ and $(\tilde{E}_n)_{6\leq n\leq 8}$ we will
denote in the same way the quiver where each edge $\bullet-\bullet$ is replaced by a pair of arrows $\bullet\leftrightarrows\bullet$.
As shown in the proof of \cite[7.2.3]{farnsteiner2006polyhedral}, the McKay quiver 
${\Upsilon}_{L(1)^{[1]}}(\hat{\G}/\hat{\G}_1)$ is isomorphic to one of the
quivers $\tilde{A}_{2np^{r-1}-1},\tilde{D}_{np^{r-1}+2}, \tilde{E}_6, \tilde{E}_7, \tilde{E}_8$,
where $r$ is the height of $\hat{\G}^0$ and $(n,p)=1$. \\
For any quiver $Q$ we denote by $Q_s$ its separated quiver. If $\{1,\dots n\}$ is the vertex set of $Q$, then $Q_s$ has $2n$ vertices
$\{1,\dots n,1',\dots n'\}$ and arrows $i\rightarrow j'$ if and only if $i\rightarrow j$ is an arrow in $Q$. 
The separated quiver of one of the quivers $\tilde{A}_{2np^{r-1}-1},\tilde{D}_{np^{r-1}+2}, \tilde{E}_6, \tilde{E}_7, \tilde{E}_8$
is the union of 2 quivers with the same underlying graph as the original quiver
and each vertex is either a source or a sink.
\begin{thm}
Let $\G$ be an amalgamated polyhedral group scheme and $\Theta$ a component of $\Gamma_s(\G)$ 
containing a $\G$-module of complexity $2$. Let $Q$ be a connected component of $\Upsilon_{L(1)^{[1]}}(\hat{\G}/\hat{\G}_1)_s$. 
Then $\Theta$ is isomorphic to $\Z[Q]$.
\end{thm}
\begin{proof}
Thanks to \cite[7.3.2]{farnsteiner2006polyhedral} all the non-simple blocks of $k\G$ are Morita equivalent
to the principal block $\cB_0(\G)$ of $k\G$. Additionally, by \cite[1.1]{farnsteiner2006polyhedral}, 
the block $\cB_0(\G)$ is isomorphic to the block $\cB_0(\hat{\G})$. Therefore it suffices to prove this result for $\G:=\hat{\G}$.\\
In view of \cite[5.6]{friedlander2005representation} and \cite[3.1]{farnsteiner2007support}, all modules belonging
to the component $\Theta$ have complexity $2$. Let $S_1,\dots,S_m$ be the simple $\G/\G_1$-modules and 
$M\in\Theta$. Due to \cite[7.4]{kirchhoff2015domestic},
there are $0\leq l\leq p-2$, $n\geq 0$ and $1\leq j\leq m$ such that $M\cong V(n,l)\otimes_kS_j$. \\
 Thanks to \cite[7.4.1]{farnsteiner2006polyhedral}, the group algebra $k\G$ is symmetric.
 Therefore the Auslander-Reiten translation $\tau_\G$ equals $\Omega_\G^2$ (see \cite[4.12.8]{benson1998representationsI}).
 Applying \ref{extend ar sequence} and \ref{decompose induced ar seq} to the almost split exact sequence
 \[\ses{V(n+2,l)}{V(n+1,l)\oplus V(n+1,l)}{V(n,l)}{}{}\]
 of $SL(2)_1$-modules yields the almost split exact sequence
 \[\ses{\tau_\G(V(n,l))}{V(n+1,l)\otimes_k \irr_{SL(2)_1}(V(n+1,l),V(n,l))}{V(n,l)}{}{}\]
 of $\G$-modules. Due to \ref{heller of weyl}, we have 
 $\tau_\G(V(n,l))\cong\Omega^2_\G(V(n,l))\cong V(n+2,l)$.
 In conjunction with \ref{irr for SL(2)} and \ref{ind is direct sum of AR-seq} we now 
 obtain the almost split exact sequence
 \[\ses{ V(n+2,l)\otimes_kS_j}{V(n+1,l)\otimes_k L(1)^{[1]}\otimes_kS_j}{V(n,l)\otimes_kS_j}{}{}.\]
 Hence $\tau_\G(V(n,l)\otimes_kS_j)\cong V(n+2,l)\otimes_kS_j$.
 Moreover, due to the decomposition $L(1)^{[1]}\otimes_kS_j\cong\bigotimes_{i=1}^ma_{ij}S_i$, there
 are $a_{ij}$ arrows $V(n+1,l)\otimes_kS_i\rightarrow V(n,l)\otimes_kS_j$ and
 $a_{ij}$ arrows $V(n+2,l)\otimes_kS_j\rightarrow V(n+1,l)\otimes_kS_i$ in $\Theta$. 
 Without loss we can now assume that $n=0$, so that $V(0,l)\otimes_kS_j$ belongs to $\Theta$.\\
 Denote by $\{1,\dots,m\}$ the vertex set of $\Upsilon_{L(1)^{[1]}}({\G}/{\G}_1)$ and by
 $\{1,\dots,m,1',\dots,m'\}$ the vertex set of its separated quiver. Let $Q$ be the connected component
 of $\Upsilon_{L(1)^{[1]}}({\G}/{\G}_1)_s$ which contains $j'$. 
 If $N$ is another module in $\Theta$, then there are $\mu\in\Z$,$0\leq \tilde{l}\leq p-2$
 and $t\in\{1,\dots,m\}$ with $N\cong V(\mu,\tilde{l})\otimes_k S_t$. 
 By the above, $M$ and $N$ can only lie in the same component if $l=\tilde{l}$.
 If $\mu=2\nu$ is even, then $\tau_\G^{-\nu}(N)\cong V(0,l)\otimes_kS_t$.
 As $M$ and $N$ are in the same component, there is a path
 \[V(0,l)\otimes_lS_j\leftarrow V(1,l)\otimes_lS_{i_1}\rightarrow V(0,l)\otimes_lS_{i_2}
 \leftarrow\dots\leftarrow V(1,l)\otimes_lS_{i_r}\rightarrow V(0,l)\otimes_lS_{t}\]
 in $\Theta$. This gives rise to a path
 \[j'\leftarrow i_1\rightarrow i_2'\leftarrow\dots\leftarrow i_r\rightarrow t'\]
 in the separated quiver $\Upsilon_{L(1)^{[1]}}({\G}/{\G}_1)_s$.
 Consequently, $t'\in Q$. Similarly, if $\mu$ is odd, we obtain $t\in Q$.\\
 Moreover, for each arrow $i\rightarrow t'$ in $Q$ and $\mu\in\Z$ we have arrows
 \[\phi_{i,t',\mu}:V(\mu+1,l)\otimes_kS_i\rightarrow V(\mu,l)\otimes_kS_t\text{ and }\]
 \[\phi_{t',i,\mu+1}:V(\mu+2,l)\otimes_kS_t\rightarrow V(\mu+1,l)\otimes_kS_i\] in $\Theta$. \\
 Now let $\psi:\Z[Q]\rightarrow \Theta$ be the morphism of stable translation quivers given by
 \[\psi(\nu,t)=V(2\nu+1,l)\otimes_kS_t\text{ for each }\nu\in\Z\text{ and }t\in\{1,\dots,m\}\]
 \[\psi(\nu,t')=V(2\nu,l)\otimes_kS_t\text{ for each }\nu\in\Z\text{ and }t'\in\{1',\dots,m'\}\]
 \[\psi((\nu,i)\rightarrow(\nu,t'))=\phi_{i,t',2\nu}\]
 \[\psi((\nu+1,t')\rightarrow(\nu,i))=\phi_{t',i,2\nu+1}.\]
One now easily checks that this is an isomorphism.
\end{proof}

\section{Quotients of support varieties and ramification}
The goal of this section is to describe a geometric connection between the tubes in the Auslander-Reiten quiver of a finite group scheme $\G$
and the corresponding tubes in the Auslander-Reiten quiver of a normal subgroup scheme $\cN$ of $\G$. 
We will see that the support variety of $\cN$ is a geometric quotient of the support variety of $\G$.
The geometric connection will then be given via the ramification indices of the quotient morphism.\\

Let $k$ be an algebraically closed field. We say that $X$ is a variety, if it is a separated reduced prevariety over $k$ and
we will identify it with its associated separated reduced $k$-scheme of finite type 
(c.f. \cite{jantzen2007representations},\cite{goertz2010algebraic}). A point $x\in X$ is always supposed to be closed
and therefore also to be $k$-rational, as $k$ is algebraically closed. Let $x\in X$, $R$ be a commutative $k$-algebra and
$\iota_R:k\rightarrow R$ be the canonical inclusion. Then we denote by $x_R:=X(\iota_R)(x)$ the image of $x$ in $X(R)$.
An action of a group scheme on a variety is always supposed to be an action via morphisms of schemes.
\begin{Def}
Let $\cH$ be a group scheme acting on a variety $X$.
\begin{enumerate}
\item A pair $(Y,q)$ consisting of a variety $Y$ and an $\cH$-invariant morphism $q:X\rightarrow Y$ is called categorical quotient
of $X$ by the action of $\cH$, if for every $\cH$-invariant morphism $q':X\rightarrow Y'$ of varieties there is a unique morphism
$\alpha:Y\rightarrow Y'$ such that $q'=\alpha\circ q$.
\item A pair $(Y,q)$ consisting of a variety $Y$ and an $\cH$-invariant morphism $q:X\rightarrow Y$ of varieties is called geometric quotient
of $X$ by the action of $\cH$, if the underlying topological space of $Y$ is the quotient of the underlying topological space of $X$
by the action of the group $\cH(k)$ and $q:X\rightarrow Y$ is an $\cH$-invariant morphism of schemes such that
the induced homomorphism of sheafs  $\mathcal{O}_Y\rightarrow q_*(\mathcal{O}_X)^{\cH}$ is an isomorphism.
\item If $x\in X$ is a point, then the stabilizer $\cH_x$ is the subgroup scheme of $\cH$ given by
$\cH_x(R)=\{g\in\cH(R)\;\vert\; g.x_R=x_R\}$ for every commutative $k$-algebra $R$.
\end{enumerate}
\end{Def}
Thanks to \cite[12.1]{mumford1974abelian}, there is for any finite group scheme $\cH$ and any quasi-projective variety $X$ 
an up to isomorphism uniquely determined geometric quotient which will be denoted by $X/\cH$.
Moreover, the quotient morphism $q:X\rightarrow X/\cH$ is finite, i.e.
there exists an open affine covering $X/\cH=\bigcup_{i\in I} V_i$ such that $q^{-1}(V_i)$ is affine and the ring homomorphism
$k[V_i]\rightarrow k[q^{-1}(V_i)]$ is finite for all $i\in I$.
\begin{example}
Let $\cH$ be a finite group scheme and $A$ be a finitely generated commutative $k$-algebra. Then the set $X:=\maxspec A$ of maximal ideals of $A$
is an affine variety. An action of $\cH$ on $X$ gives $A$ the structure of a $k\cH$-module algebra over the Hopf algebra $k\cH$. This means that
$A$ is an $\cH$-module such that $h.(ab)=\sum_{(h)}(h_{(1)}.a)(h_{(2)}.b)$ and $h.1=\epsilon(h)1$ where $\epsilon:k\cH\rightarrow k$
is the counit of $k\cH$. Then the set \[A^\cH:=\{a\in A\;\vert\;h.a=\epsilon(h)a\text{ for all }h\in k\cH\}\] of $\cH$-invariants
is a subalgebra of $A$.
As $\cH$ is finite, the subalgebra is also finitely generated and $X/\cH=\maxspec A^\cH$. \\
Let $A=\bigoplus_{n\geq 0}A_n$ be additionally graded with $A_0=k$ and $A_+=\bigoplus_{n> 0}A_n$ be its irrelevant ideal. 
Then the set $X=\proj A$ of maximal homogeneous ideals which do not contain $A_+$ is a projective variety.
If $\cH$ acts on $X$, then $X/\cH=\proj A^\cH$.
\end{example}
Let $X$ be a variety with structure sheaf $\mathcal{O}_X$. For a point $x\in X$ we denote by $\mathcal{O}_{X,x}$ the local ring at the point $x$.
We say that a point $x\in X$ is simple, if its local ring $\cO_{X,x}$ is regular.
If $R$ is a commutative local ring, we will denote by $\hat{R}$ its completion at its unique maximal ideal.
Let $V$ be a $k$-vector space of dimension $n$ with basis $b_1,\dots,b_n$. Let $t_1,\dots,t_n$ be the corresponding dual basis of $V^*$.
We define the ring of polynomial functions of $V$ as $k[V]:=k[t_1,\dots,t_n]$. Its completion $k[[t_1,\dots,t_n]]$ at the maximal ideal
$(t_1,\dots,t_n)$ will be denoted by $k[[V]]$.\\
For future reference we will recall the following facts:

\begin{rmk}
Let $X$ be an $n$-dimensional variety and $x\in X$.
\begin{enumerate}[label={(\roman*)},ref={\thecor~(\roman*)}]
\item  A point $x\in X$ is simple if and
only if $\hat{\mathcal{O}}_{X,x}\cong k[[x_1,\dots,x_n]]$. \label{local ring of smooth point}
\item Let $\cH$ be a finite group scheme acting on $X$ and assume that there is a geometric quotient $(Y,q)$ of this action.
Then $\hat{\mathcal{O}}_{Y,q(x)}\cong(\hat{\mathcal{O}}_{X,x})^{\cH_x}$ (\cite[Exercise 4.5(ii)]{benmoonen}). \label{local ring for quotient}
\item Let $\cH$ be a linearly reductive group scheme acting on $k[[x_1,\dots,x_n]]$ via algebra automorphisms. Then there is an $n$-dimensional
$\cH$-module $V$ such that there is an $\cH$-equivariant isomorphism $k[[V]]\cong k[[x_1,\dots,x_n]]$ of $k$-algebras. 
(c.f. \cite[Proof of 1.8]{satriano2012chevalley})
\label{linear action on formal series ring}
\end{enumerate}
\end{rmk}
%
\begin{Def}
Let $f:X\rightarrow Y$ be a finite morphism, $x\in X$ and $y=f(x)$. Let $\fm_y$ be the maximal ideal of the local ring $\cO_{Y,y}$
Then $e_x(f):=\dim_k\cO_{X,x}/\fm_y\cO_{X,x}$ is called the ramification index of $f$ at $x$.
\end{Def}
As the morphism $f$ is finite, the induced homomorphism $\cO_{Y,y}\rightarrow\cO_{X,x}$ endows $\cO_{X,x}$ with the structure
of a finitely generated $\cO_{Y,y}$-module. Therefore the number $e_x(f)$ is finite.\\

For $n\in\N$, let $\mu_{(n)}$ be the finite group scheme with $\mu_{(n)}(R)=\{x\in R\;\vert\;x^n=1\}$
for every commutative $k$-algebra $R$. For a finite group scheme $\cH$ we denote by $\vert\cH\vert:=\dim_kk\cH$ its order.
\begin{lem}\label{ramification and stabilizer}
Let $\cH$ be a finite linearly reductive group scheme which acts faithfully on a $1$-dimensional irreducible quasi-projective variety $X$.
Denote by $q:X\rightarrow X/\cH$ the quotient morphism.
If $x\in X$ is a simple point, then $\cH_x\cong\mu_{(n)}$ for some $n\in\N$ and $e_x(q)=\vert \cH_x\vert$. 
\end{lem}
\begin{proof}
Since $x$ is a simple point and $X$ is $1$-dimensional, \ref{local ring of smooth point} yields a $\cH_x$-equivariant
isomorphism $\hat{\cO}_{X,x}\cong k[[T]]$. By \ref{linear action on formal series ring}, we can assume that the one-dimensional $k$-vector space
$\langle T\rangle_k$  is an $\cH_x$-module. 
As $X$ is irreducible, the field of fractions of $\cO_{X,x}$ is the function field $k(X)$ of $X$. If $\cK$ is the kernel of the action of $\cH_x$
on ${\cO}_{X,x}$, then it also acts trivially on $k(X)$ and therefore also on $X$. As $\cH$ acts faithfully on $X$, it follows that
$\cK$ is trivial. Therefore, $\langle T\rangle_k$ is a faithful $\cH_x$-module and we can assume $\cH_x=\mu_{(n)}\subseteq GL_1$ where $n=\vert \cH_x\vert$.
As $\cH_x$ acts via algebra automorphisms, this yields $k[[T]]^{\cH_x}=k[[T^n]]$. By \ref{local ring for quotient}, we have
$\hat{\cO}_{X/\cH,q(x)}\cong\hat{\cO}_{X,x}^{\cH_x}$.
As a result, we obtain
\[e_x(q)=\dim_k \hat{\cO}_{X,x}/\fm_{q(x)}\hat{\cO}_{X,x}=\dim_k k[T]/(T^n)=n=\vert \cH_x\vert.\]
\end{proof}

Let $k$ be an algebraically closed field of characteristic $p>0$ and $\mathcal{G}$ be a finite group scheme. Denote by $\V_\G$ the cohomological support variety of $\mathcal{G}$. By definition
$\mathcal{V}_{\mathcal{G}}=\maxspec H^\bullet(\mathcal{G},k)$ is the spectrum of maximal ideals of the
even cohomology ring $H^\bullet(\mathcal{G},k)$ of $\mathcal{G}$ (in characteristic $2$ one takes instead the whole cohomology ring).
The projectivization of the cohomological support variety will be denoted by $\mathbb{P}(\V_\G)$.
For every finite-dimensional $\G$-module $M$ there is a natural homomorphism $\Phi_M:H^\bullet(\G,k)\rightarrow\Ext_{\G}^*(M,M)$ of graded $k$-algebras.
The cohomological support variety of $M$ is then defined as the subvariety
$\mathcal{V}_{\mathcal{G}}(M)=\maxspec (H^\bullet(\mathcal{G},k)/\ker\Phi_M)$ of $\mathcal{V}_\G$. \\
For a subgroup scheme $\mathcal{H}$ of $\G$ let $\iota_{*,\mathcal{H}}:\P(\V_{\mathcal{H}})\rightarrow\P(\V_{\G})$ 
be the morphism which is induced by the canonical inclusion $\iota:k\mathcal{H}\rightarrow k\mathcal{G}$. 
If $\cN$ is a normal subgroup scheme of $\G$ then
$\G/\cN$ acts via automorphisms of graded algebras on $H^\bullet(\cN,k)$ and therefore on $\P(\V_\cN)$.
\begin{prop}\label{quotient of support variety}
Let $\G$ be a finite group scheme and $\cN$ be a normal subgroup scheme of $\G$ such that $\G/\cN$ is linearly reductive.
Then $(\P(\V_\G),\iota_{*,\cN})$ is a geometric quotient for the action of $\G/\cN$ on $\P(\V_{\cN})$.
\end{prop}
\begin{proof}
Since $\G/\cN$ is linearly reductive, the Lyndon-Hochschild-Serre spectral sequence
(\cite[I.6.6(3)]{jantzen2007representations}) yields an isomorphism
$\iota^\bullet:H^\bullet(\G,k)\rightarrow H^\bullet(\cN,k)^{\G/\cN}$. 
Therefore, $(\P(\V_\G),\iota_{*,\cN})$ is a geometric quotient for the action of $\G/\cN$ on $\P(\V_{\cN})$.
\end{proof}
%
%
Let $H$ be a group. A $k$-algebra $A$ is called strongly $H$-graded if it admits a decomposition $A=\bigoplus_{g\in H}A_g$ such that
$A_gA_h=A_{gh}$ for all $g,h\in H$. If $U\subseteq H$ is a subgroup, the subalgebra $A_U:=\bigoplus_{g\in U}A_g$
is a strongly $U$-graded $k$-algebra. Let $N\subseteq H$ be a normal subgroup of $H$. Then $A$ can be regarded as a strongly $H/N$-graded
$k$-algebra via $A_{gN}:=\bigoplus_{x\in gN}A_x$ for all $g\in H$. \\
Let $M$ be an $A_1$-module. For $g\in H$ we denote by $M^g$ the $A_1$-module with same underlying space $M$ and $A_1$-action twisted by $g^{-1}$.
We will call the subgroup $G_M:=\{g\in H\;\vert\;M^g\cong M\}$ the stabilizer of $M$.

Let $\G$ be a finite group scheme, $\cN\subseteq \G$ be a normal subgroup scheme with $\G^0\subset \cN$ and set $G:=(\G/\cN)(k)$.
Due to the decomposition $\G\cong\G^0\rtimes\G_{red}$, the group algebra $k\G$ is isomorphic to the skew group algebra $k\G^0*\G(k)$.
As $\G^0\subseteq \cN$, we therefore obtain that $k\G$ has the structure of $G$-graded $k$-algebra with $(k\G)_1=k\cN$.
Let $M$ be an $\cN$-module. Then the Hopf-subalgebra $(k\G)_{G_M}$ of $k\G$ determines a unique subgroup scheme $\G_M$ of $\G$ with $k\G_M=(k\G)_{G_M}$.\\

The group $G$ acts on the module category $\mod \cN$ via equivalences of categories
$\mod \cN\rightarrow\mod \cN, M\mapsto M^g$ for $g\in G$. Since these equivalences commute with the Auslander-Reiten translation of $\Gamma_s(\cN)$,
each $g\in G$ induces an automorphism $t_g$ of the quiver $\Gamma_s(\cN)$. Therefore, 
$G$ acts on the set of components of $\Gamma_s(\cN)$. For a component $\Theta$ we write $\Theta^g=t_g(\Theta)$ and let 
$G_\Theta=\{g\in G\:\vert\: \Theta^g=\Theta\}$ be the stabilizer of $\Theta$.
If $M$ is an $\cN$-module which belongs to the component $\Theta$, then $G_M\subseteq G_\Theta$.
As above, there is a unique subgroup scheme $\G_\Theta\subseteq\G$ with $k\G_\Theta=(k\G)_{G_\Theta}$.
\\
Now let $N$ be an indecomposable non-projective $\cN$-module and $\Xi$ the corresponding component in $\Gamma_s(\cN)$. Assume there is an
indecomposable non-projective direct summand $M$ of $\ind_\cN^\G N$ and let $\Theta$ be the corresponding component in $\Gamma_s(\G)$.
\\

A stable translation quiver with tree class $A_\infty$ has for each vertex $M$ only one sectional path to the end of the component
(\cite[(VII)]{auslander1997representation}). The length of this path is called the quasi-length $\ql(M)$ of $M$.
A stable translation quiver of the form $\Z[A_\infty]/(\tau^n)$, $n\geq 1$,
is called a tube of rank $n$. These components contain for each $l\geq1$ exactly $n$ modules of quasi-length $l$.
Tubes of rank $1$ are also called homogeneous tubes and all other tubes are called exceptional tubes.\\
If $A$ and $B$ are $\G$-modules which belong to the same AR-component $\Upsilon$, then $\V_\G(A)=\V_\G(B)$
(c.f. \cite[3.1]{farnsteiner2007support}). Therefore we can define $\V_\G(\Upsilon):=\V_\G(A)$ for some $\G$-module $A$ belonging to $\Upsilon$.
If $\Upsilon$ is a tube, then $\vert\P(\V_\G(\Upsilon))\vert=1$ (c.f. \cite[3.3(3)]{farnsteiner2007support}). 
Thanks to \cite[5.6]{friedlander2005representation}, we have $\iota_{*,\cN}^{-1}(\P(\V_\G(M)))=\P(\V_\cN(\res_\cN^\G M))$.
By \cite[4.5.8]{marcus1999representation}, the module $\res_\cN^\G M$ has an indecomposable direct summand which
belongs to $\Xi$. Hence, 
\[x_\Xi\in\iota_{*,\cN}^{-1}(\P(\V_\G(M)))=\iota_{*,\cN}^{-1}(x_\Theta),\]
so that $\iota_{*,\cN}(x_\Xi)=x_\Theta$.
\begin{prop}
Let $\Xi$ be a tube of rank $n$ and $\Theta$ be a tube of rank $m$. Assume that $\G^0\subseteq \cN$ and set
$G:=(\G/\cN)(k)$. Moreover, assume the following
\begin{enumerate}
	\item $\G/\cN$ is linearly reductive, i.e. $p$ does not divide the order of $G$,
	\item $G$ acts faithfully on $\P(\V_\cN)$, 
	\item the variety $\P(\V_\cN)$ is one-dimensional and irreducible,
	\item $x_\Xi$ is a simple point of $\P(\V_\cN)$, and
	\item all modules belonging to $\Xi$ are $\G_\Xi$ stable.
\end{enumerate}
Then $m\leq e_{x_\Xi}(\iota_{*,\cN})n$.
\end{prop}
\begin{proof}
	By the above, the group algebra $k\G$ is a strongly $G$-graded $k$-algebra with $(k\G)_1=k\cN$.
	As $G$ acts faithfully on the one-dimensional irreducible variety $\P(\V_\cN)$ with simple point $x_\Xi$, 
	we obtain due to \ref{ramification and stabilizer} that the stabilizer $G_{x_\Xi}$ is a cyclic group and 
	we have the equality $e_{x_\Xi}(\iota_{*,\cN})=\vert G_{x_\Xi}\vert$.
	Now the assertion follows directly from \cite[6.3(d)]{kirchhoff2015domestic}.
\end{proof}
Let $\cZ$ be the center of the group scheme $SL(2)$ and $\tilde{\G}$ a finite linearly reductive subgroup scheme 
of $SL(2)$ with $\cZ\subseteq\tilde{\G}$. 
As mentioned in section 2, every domestic finite group scheme can be associated to one of the
amalgamated polyhedral group schemes $(SL(2)_1\tilde{\G})/\cZ$. There are 5 different possible classes for the choice of $\tilde{\G}$ 
(for details see \cite[3.3]{farnsteiner2006polyhedral}):
\begin{enumerate}
\item binary cyclic group scheme,
\item binary dihedral group scheme,
\item binary tetrahedral group scheme,
\item binary octahedral group scheme and
\item binary icosahedral group scheme.
\end{enumerate}
The latter three are always reduced group schemes.
\begin{prop}
	Let $\G$ be an amalgamated polyhedral group scheme, $\cN:=\G_1$ its first Frobenius kernel and $\Theta$ be a tube. 
	Then $\Theta$ has rank $e_{x_\Xi}(\iota_{*,\cN})$.
\end{prop}
\begin{proof}
	If $\tilde{\G}$ is a non-reduced binary dihedral group scheme, then one obtains this result directly from \cite[7.6]{kirchhoff2015domestic}
	by comparing the numbers. With the same methods as in the proof of \cite[7.6]{kirchhoff2015domestic} one can classify the indecomposable
	modules for the amalgamated cyclic group scheme. Again one obtains this result by comparing numbers. The other cases are either proved
	in the same way or with the following arguments:\\
	We can assume that $\tilde{\G}$ is reduced. Let $M$ be an indecomposable $\G$-module which belongs to $\Theta$ and
	$U$ be an indecomposable direct summand of $\res_{\G_\Xi}^\G M$ with $\ind_{\G_\Xi}^\G U=M$. Denote by $\Lambda$ the Auslander-Reiten component of $\Gamma_s(\G_\Xi)$
	which contains $U$.	Then \cite[4.5.10]{marcus1999representation} yields, that
	$\ind_{\G_\Xi}^\G:\Lambda\rightarrow\Theta$ is an isomorphism of stable translation quivers. Now the assertion follows from the cyclic case,
	as $\G_\Xi(k)$ is a cyclic group.
\end{proof}
\subsubsection*{Acknowledgement}
The results of this article are part of my doctoral thesis, which I am currently writing at the University
of Kiel. I would like to thank my advisor Rolf Farnsteiner for his continuous support as well as for helpful remarks.
Furthermore, I thank the members of my working group for proofreading.

\markboth{\nomname}{\nomname}

\footnotesize
\bibliographystyle{abbrv}
\bibliography{bib}
\noindent
\textsc{Christian-Albrechts-Universität zu Kiel, Ludewig-Meyn-Str. 4, 24098 Kiel, Germany}
\textit{E-mail address:} \textsf{kirchhoff@math.uni-kiel.de}
\end{document}